\documentclass[12pt,amstex,leqno]{amsart}
\usepackage{latexsym}
\usepackage{amsthm}
\usepackage{amsmath}

\usepackage{amsfonts}
\usepackage{amssymb}
\usepackage{appendix}
\input xy
\xyoption{all}

\swapnumbers

\newtheorem{theorem}{Theorem}[section]
\newtheorem{lemma}[theorem]{Lemma}

\newtheorem{prop}[theorem]{Proposition}

\newtheorem{corollary}[theorem]{Corollary}

\theoremstyle{definition}
\newtheorem{definition}[theorem]{Definition}

\newtheorem{example}[theorem]{Example}
\newtheorem{remark}[theorem]{Remark}
\newtheorem{notation}[theorem]{Notation}

\newtheorem{open}[theorem]{Open Problem} 
\newtheorem{conclusion}[theorem]{Conclusion} 

\numberwithin{equation}{section}

\renewcommand\Vec{\operatorname{\bf Vec}}

\newcommand\op{\operatorname{op}}

\newcommand\Set{\mathbf{Set}}

\newcommand\Top{\operatorname{\bf Top}}

\newcommand\id{\operatorname{id}}

\newcommand\ca{\mathcal {A}} 
\newcommand\cb{\mathcal {B}}

\newcommand\cd{\mathcal {D}}

\newcommand\cg{\mathcal {G}}

\newcommand\ck{\mathcal {K}}
\newcommand\cl{\mathcal {L}}

\newcommand\cu{\mathcal {U}}
 
\newcommand\du{\mathbb U}

\newcommand\G{\mathcal G}

\newcommand\N{\mathbb N}

\date{December 21, 2018}

\begin{document}

\author[Ad\'amek]{J. Ad\'amek}

\address{\newline Department of Mathematics,\newline Faculty of Electrical Engineering,\newline
Czech Technical University in Prague\newline Czech Republic} 
\email{j.adamek@tu-bs.de}

\author[Brooke-Taylor]{A. Brooke-Taylor}
\address{\newline School of Mathematics,\newline
University of Leeds
\newline Leeds, LS2 9JT
\newline United Kingdom}
\email{a.d.brooke-taylor@leeds.ac.uk}

\author[Campion]{T. Campion}
\address{\newline Department of Mathematics
\newline University of Notre Dame
\newline 255 Hurley, Notre Dame, IN 4655
\newline USA}
\email{tcampion@nd.edu}

\author[Positselski]{L. Positselski}
\address{\newline  Institute of Mathematics,
\newline Czech Academy of Sciences,
\newline \v Zitn\'a~25, 115~67 Prague~1,
\newline Czech Republic}
\email{positselski@yandex.ru}

\author[Rosick\'y]{J. Rosick\'y} 
\address{\newline Department of Mathematics and Statistics,\newline
Masaryk University,\newline
 Kotl\'a\v rsk\'a 2, 611 37 Brno,\newline Czech Republic}
\email{rosicky@math.muni.cz}

\dedicatory{Dedicated to the memory of V{\v{e}}ra Trnkov{\'{a}}, a great teacher and dear friend}

\title{Colimit-Dense Subcategories}

\thanks{L. Positselski was supported by research plan RVO: 67985840 and J. Rosick\' y 
	by the Grant agency of the Czech republic under the grant P201/12/G028.} 

\begin{abstract}
	Among cocomplete categories, the locally presentable ones can be defined as those with a strong generator consisting of presentable objects. Assuming Vop{\v{e}}nka's Principle, we prove that a cocomplete  category is locally presentable  iff it has a colimit dense subcategory and a generator consisting of presentable objects. We further show that a $3$-element set is colimit-dense in $\Set^{\op}$, and spaces of countable dimension are colimit-dense in $\Vec^{op}$.
	
\end{abstract}

\maketitle

\section{Introduction}

Our paper is devoted to the question of existence of \textit{colimit-dense} subcategories of a given category $\ck$, i.e.,  small, full subcategories such that every object of $\ck$ is a colimit of a diagram in that subcategory. We show e.g. that a set of $3$ elements forms a colimit-dense subcategory of the dual of $\Set$.
In contrast, finitely-dimensional vector spaces are not colimit-dense in the dual of $\Vec$ - but spaces of countable dimension are.

 Recall that a small, full subcategory $\cg$ of $\ck$ is called \textit{dense} if every object $X$ of $\ck$ is the canonical colimit of objects of $\cg$. That is, the diagram 
 $$D_X:\cg/X\to\ck$$
  assiging to every object $g:G\to X$ its domain has colimit $X$ with the canonical colimit cocone. Whether or not $\Set^{\op}$ has a dense subcategory depends on the following set-theoretical assumption:
\begin{enumerate}
\item[(M)] There exists a cardinal $\lambda$ such that every $\lambda$-complete ultrafilter is principal.
\end{enumerate}
Indeed, (M) is equivalent to $\Set^{\op}$ having a dense subcategory, as proved by Isbell \cite{I}. We present a short proof in Section 3, and discuss the analogous result for the dual of $\Vec$, the category of vector spaces over a given field, in Section 4.
 
 This is closely related to a  joint paper of two of the co-authors with V{\v{e}}ra Trnkov{\'{a}} \cite{RTA}. There we investigated properties of locally presentable categories depending on the validity of \textit{Vop{\v{e}}nka's Principle}. This principle states that there is no rigid proper class of graphs, i.e., a class where the only homomorphisms are the identity endomorphisms. This is a famous set-theoretical statement which implies $\neg$(M). And since (M) is consistent with set theory, the negation of Vop{\v{e}}nka's Principle is consistent as well. On the other hand, Vop{\v{e}}nka's Principle is consistent with any set theory in which huge cardinals exist. These and other facts can be found e.g.\ in Chapter~6 of \cite{AR}. In the above joint paper \cite{RTA} we proved that if Vop{\v{e}}nka's Principle is assumed, the following hold:
 
 \begin{enumerate}
 \item[(a)] every full subcategory of a locally presentable  category is \textit{bounded}, i.e., it has a dense subcategory,  
 \item[(b)] every full subcategory of a locally presentable  category closed under limits is reflective and locally presentable,
 \end{enumerate}
 and
 \begin{enumerate}
 \item[(c)] every cocomplete bounded category is locally presentable.
 \end{enumerate}
 Furthermore, each of the statements (a)--(c) was also proved to imply Vop{\v{e}}nka's Principle.
  
  Unfortunately, one of the statements of \cite{RTA} turns out to be incorrect, and one of our aims is to repair that statement. According to Theorem~9 of \cite{RTA} Vop{\v{e}}nka's Principle implies that every category with a colimit-dense subcategory is bounded. This result also appeared as Theorem 6.35 in \cite{AR}. However, $\Set^{\op}$ is a counter-example, as mentioned above. We provide a correction by proving a weaker statement.
Let us call  a generator of a category \textit{presentable} if it consists of presentable objects. 
The weaker statement proved in Section 2 is the following

\begin{theorem}
Vop{\v{e}}nka's Principle implies that every cocomplete category having both a colimit-dense subcategory and a presentable generator is bounded. 
\end{theorem}

 Recall from \cite{AR} that locally presentable categories  are precisely the cocomplete categories with a presentable strong generator. However, a presentable generator is not sufficient: the category $\Top$ of topological spaces has a presentable generator $\{1\}$, but it is not locally presentable. And a strong generator (or even a colimit-dense subcategory) is also not sufficient: every complete lattice is colimit-dense in itself, but not every one  is algebraic (= locally presentable). The combination as in the above theorem is sufficient under Vop{\v{e}}nka's Principle -- and we do not know at present whether that assumption is needed:
 
 \begin{open}
  If every category with a presentable generator and a colimit-dense subcategory is bounded, does Vop{\v{e}}nka's Principle follow?
  \end{open}
 
 As a byproduct of our study, we present parallel proofs for two classical results: one is that the codensity monad of the embedding of finite sets in $\Set$ is the ultrafilter monad. This was already proved by Kennison and Gildenhuis in 1971, see \cite{KG}, and a nice new proof is due to Leinster \cite{L}, which we recall in Section 3.
 The other result is that the codensity monad of the embeding of finite-dimensional vector spaces into $\Vec$ is the double-dualization monad. This is due to Leinster \cite{L} but the (almost equivalent) fact that the double-dual of a vector space is its profinite completion is well-known, see for example \cite{Be} or \cite{Ba}. We present a new proof, based on Leinster's ideas, in Section 4.

 \section{Colimit-Dense Subcategories in General}
 
 Recall that an object $K$ of a category $\ck$ is \textit{presentable} if
its hom-functor preserves $\lambda$-filtered colimits for some regular cardinal $\lambda$.
 Recall further that a \textit{generator} is a set $\G$ of objects whose hom-functors are collectively faithful; it is considered as a full subcategory of $\ck$.
 A generator $\G$ is called
 \begin{enumerate}
 \item[(a)] \textit{presentable} if all objects of it are presentable, and
 \item[(b)] \textit{strong} if every monomorphism $L \rightarrowtail K$ such that all morphisms from $G\in \G$ to $K$ factorize through it is invertible.
 \end{enumerate}
 
 \begin{definition}\label{consistent}
  A small, full subcategory $\cg$ of the category $\ck$ is called \textit{colimit-dense} if every object $K$ 
of $\ck$ is a colimit of some diagram in $\G$. It is called \textit{consistently colimit-dense} if it contains a generator $\cg_0$ of $\ck$ for which every object of $\ck$ is a colimit of a diagram in $\cg$ such that all hom-functors of objects of $\cg_0$ preserve this colimit.
\end{definition}

It is easy to see that the following implications
$$
\mbox{dense} \quad \Rightarrow \quad \mbox{colimit-dense} \quad \Rightarrow\quad \mbox{strong generator}
$$
hold. None  can be reverted: $K$ is clearly colimit-dense in the category $\Vec$ of vector spaces over $K$, but it is not dense.
 (In contrast, $K\times K$ is dense.) All sets of power at most $2$ form a strong generator of $\Set^{op}$, which is demonstrated in Remark \ref{colimdense1} 
 not to be colimit-dense.
 And in the category of compact Hausdorff spaces the space $1$ forms a strong generator which is not colimit-dense. Finally, not every generator is strong, e.g.\ the discrete one-element graph is a generator of the category of graphs that is not strong.
 
Recall from the Introduction that a category is said to be \textit{bounded} if it has a dense generator. And that it is locally presentable iff it is cocomplete and has a strong, presentable generator. 
 
\begin{theorem}\label{correction0}
Assume Vop\v enka's Principle. Then a  category $\ck$ with a consistently colimit-dense subcategory is bounded. 
\end{theorem}
\begin{proof}
The proof of Theorem 9 in \cite{RTA} is valid with one exception:
the calculations in Part (ii) on page 168 are incorrect. (The same is true about item (a) of the proof of Theorem 6.35 of \cite{AR}.)  We correct those calculations as follows:

Let $\G$ be a consistently colimit-dense generator of a category $\ck$, and let $\G_0$ be the generator consisting of all members $G \in \cg$ such that the hom-functor of $G$ preserves the given colimits $K = colim B_K$ of diagrams $B_K$ in $\G$ (for all objects $K$). Take
the canonical functor $E:\ck\to\Set^{\G_0^{\op}}$ assigning to an object $K$ the domain-restriction of $\ck(-,K)$ to the dual of $\cg_0$. 
This functor is faithful (because $\G_0$ is a generator) and it preserves colimits of the given diagrams $B_K$ for all objects $K$. For every object $G$ of $\cg_0$ we obtain a presheaf $EG$ on $\G_0$
and define $U_0 : \Set^{\G_0^{\op}} \to \Set$ as the coproduct of the corresponding representable functors:

$$
U_0=\coprod_{G\in\G_0}\Set^{\G_0^{\op}}(EG,-).
$$
Then $U_0$ is faithful and preserves colimits of the given diagrams $B_k$ postcomposed by $E$, therefore, the composite functor
$$U_0E:\ck\to\Set.
$$ 
is faithful and preserves the colimits of all the diagrams $B_K$.
The above mentioned calculations in (ii) (or in item (a), resp.) are correct when the functor U there is substituted by $U_0E:\ck\to\Set.$
\end{proof}

\begin{corollary}\label{correction2}
	Assume Vop{\v{e}}nka's Principle. A category is locally presentable iff it is cocomplete  and has a colimit-dense subcategory and a presentable generator.
\end{corollary}

\begin{proof}
Let $\ck$ be a cocomplete category with a presentable generator $\G_p$ and a colimit-dense one $\G_c$. There is a regular cardinal $\lambda$ such that $\ck(G,-)$ preserves $\lambda$-filtered colimits for every $G\in\G_p$. Let $\G$ be the closure of $\G_c$ under 
$\lambda$-small colimits (i.e.,  colimits over diagrams having less than $\lambda$ morphisms).
For every object $K$ we have the chosen diagram $B_K$ in  $\G_c$ and we denote by
$B'_K$ its extension to $\G$ obtained by a free completion of the domain of $B_K$ under $\lambda$-small colimits.
This is a $\lambda$-filtered diagram in $\G$, thus hom-functors of objects of  $\G_p$
 preserve the colimit of $B'_K$.
Therefore, $\G \cup \G_p$ is a clearly a consistently colimit-dense subcategory. Thus the result follows from Theorem \ref{correction0}.
\end{proof}

%

However, Vop{\v{e}}nka's Principle  does \textit{not} imply  that a cocomplete category with a colimit-dense subcategory is locally presentable (that is, Theorem~9 of \cite{RTA} and Theorems 6.35 and 6.37 of \cite{AR} are false). We will see this in the next Section. The next example shows that  the converse implication holds:

\begin{example}\label{ex2}
 Assuming the negation  of Vop{\v{e}}nka's Principle (which, recall, is consistent with set theory) the following 
 category $\ck$ was proved in \cite[Example 1]{AHR} to be cocomplete and non-bounded, although it has a finite  colimit-dense subcategory. (Unfortunately, $\ck$ does not have a presentable generator.) Let $\cl$ be a large rigid class
 of graphs. Objects of $\ck$ are triples $(X,Y,\alpha)$ where $X$ is a set, $Y\subseteq X$ and $\alpha$ is a graph, i.e.\ a binary relation, on $Y$. Morphisms
 $$
 f\colon (X,Y,\alpha) \to (X', Y', \alpha')
 $$
 are functions $f\colon X\to X'$ such that $f$ extends a graph homomorphism from $(Y, \alpha)$ to $(Y', \alpha')$ and for every $x\in X\setminus Y$ we have
 $$
 f(x) \in \big(X' \setminus Y'\big) \cup \bigcup_{L\in \cl} \bigcup_{h\colon L\to (Y', \alpha')} h[L]\,.
 $$
\end{example}
 
\begin{remark}\label{correction1}
The assumptions of Theorem \ref{correction0} can be weakened to the existence of a colimit-dense subcategory $\G$ and a faithful functor
$U:\ck\to\Set$ preserving colimits of the diagrams $B_K$ used in expressing $\ck$-objects by $\G$-objects. This follows from our proof: use $U$ in place of $U_0E$.
\end{remark}
 
\section{Colimit-Density in $\Set^{\op}$}
 
Recall that an ultrafilter $\cu$ on a set $X$ is called $\lambda$-\textit{complete}, for an infinite cardinal $\lambda$, if it is closed under intersections of less than $\lambda$ members. This can be expressed via $\lambda$-\textit{partitions} of $X$, i.e., partitions $(X_i)_{i \in n}$ of $X$ into $n <\lambda$ nonempty subsets: an ultrafilter is $\lambda$-complete iff for every $\lambda$-partition $(X_i)$ it contains precisely one member $\alpha(X_i)$. This gives rise to a function $\alpha$ from $\lambda$-partitions of a set to its power-set
such that $\alpha(X_i)_{i \in n} = X_j$ for some $j \in n$.

\begin{lemma}[Galvin, Horn \cite{GH}]\label{ultra0} For every set $X$ a collection of subsets is a $\lambda$-complete ultrafilter	
	iff it contains a unique member in every $\lambda$-partition of $X$. That is, the above passage $\cu \mapsto \alpha$ is bijective.
\end{lemma}

\begin{remark}\label{ultra}
\cite{GH} also shows that the unique choice corresponding to a $\lambda$-complete ultrafilter $\cu$ is \textit{coherent}, i.e., 
 if $Q_2$ is coarser than $Q_1$ then the chosen member of $Q_2$ contains that of $Q_1$.
\end{remark}

Let 
$$U_\lambda(X)$$ 
denote the set of all $\lambda$-complete ultrafilters on a set $X$.

For every, possibly finite, $\lambda$ denote by
$$
\Set_\lambda
$$
the full subcategory of $\Set$ consisting of sets of cardinality less that $\lambda$. 

\begin{lemma}\label{ultra3}
For every set $X$, $U_\lambda(X)$ is the limit of the canonical diagram formed by all $\lambda$-partitions of $X$.
\end{lemma} 
\begin{proof}
Consider the diagram $D_X: X/\Set_\lambda \to \Set$ assigning to every object of $X/\Set_\lambda$ its codomain.
Every mapping $f:X\to Z$ with $|Z|<\lambda$ induces a $\lambda$-partition $Q_f$ of $X$. Another such mapping $f_2:X\to Z_2$
factorizes through $f_1:X\to Z_1$ if and only if $Q_{f_2}$ is coarser than $Q_{f_1}$. Thus, the limit of the  diagram $D_X$ consists of coherent choices of elements of $\lambda$-partitions $Q_f$. Hence, our lemma follows from \ref{ultra0} and \ref{ultra}. 
\end{proof}

\begin{corollary}[Isbell \cite{I}]\label{isbell}
Let $\lambda$ be an infinite cardinal. Then $\Set^{\op}_\lambda$ is dense in $\Set^{\op}$ iff every $\lambda$-complete ultrafilter is principal. 
\end{corollary}

Indeed, the factorizing mapping from $X$ to the limit of the canonical diagram of all $\lambda$-partitions sends an element $x$
to the principal ultrafilter generated by $x$. Thus $\Set_\lambda$ is limit-dense in $\Set$ iff every $\lambda$-complete ultrafilter is principal.

\begin{corollary}[Isbell \cite{I}]\label{isbell1}
$\Set^{\op}$ is bounded iff (M) holds. 
\end{corollary}

Indeed, suppose $\Set$ has a codense subcategory $\cg$ and $\lambda$ is an infinite cardinal larger than $|G|$ for each $G\in\cg$. Then $\cg$ is cofinal in 
$\Set_\lambda$ and thus the above diagram $D_X:X/\cg\to\Set$ is cofinal in the corresponding diagram w.r.t.\ $\Set_\lambda$. Thus $\Set_\lambda$
is codense in $\Set$, and our result follows from \ref{isbell}.
 
 \begin{remark}\label{monad}
 	
 	(a) Recall that the ultrafilter functor $U:\Set\to\Set$ assigns to every set $X$ the set $U(X)$ of all
 	ultrafilters on it and to every mapping $f:X\to Y$ the mapping
 	$$
 	U f: \cu\mapsto \{A\subseteq Y; f^{-1}(A)\in\cu\}.
 	$$
 	This functor yields a monad $\du$ with the unit $\eta_X:X\to U(X)$ given by principal ultrafilters.
 	Indeed, $U$ carries a unique structure of a monad, as proved by B\"orger \cite{B}. This is based on the fact that $U$
 	is terminal in the category of all set functors preseving finite coproducts.
 	
 	The subfunctor $U_\lambda$ of $U$ of all $\lambda$-complete ultrafilters also carries a unique structure of a monad.
 	B\"orger's proof is easily adapted: $U_\lambda$ is terminal in the category of all set functors preserving coproducts
 	of size smaller than $\lambda$. We thus obtain a submonad $\du_\lambda$ of $\du$.
 	
 	(b) Recall further the concept of the \textit{codensity monad} of a small, full subcategory $\cg$ of a complete catory $\ck$: this is the
 	monad given by the left Kan extension of the embedding $\cg \hookrightarrow \ck$ over itself. Explicitly,  this is the monad $(T, \mu,\eta)$
 	where $T$ assigns to an object $K$  the limit of the diagram $D_K: K/\cg \to \ck$ assigning to every object $g:K \to G$ its codomain (with a limit cone $\pi_g: TK \to G$). To a morphism $k:K \to L$ this functor assigns the unique morphism $Tk$ with 
 	$$
 	\pi_g.Tk=\pi_{k.g}
 	$$
 	for all $g:L \to G$ in $L/\cg$. The monad unit has components $\eta_k$ determined by $\pi_g .\eta_k=id$ for all $g \in K/\cg$.
 \end{remark}

\begin{corollary}\label{monad1}
The monad $\du_\lambda$ is the codensity monad of the embedding
$$
\Set_\lambda \hookrightarrow \Set.
$$
\end{corollary}

Indeed, the formula for the codensity monad in (b) above demonstrates that $T$ agrees with $U_\lambda$. Thus the corollary follows from (a).

\begin{remark}\label{monad2}
The special case $\lambda=\omega$ is the classical result of Kennison and Gildenhuis \cite{KG} that the ultrafilter monad is the codensity monad of the embedding of finite sets into $\Set$. The above proof for this case was presented by Leinster \cite{L}.
\end{remark} 

Surprisingly, $\Set^{\op}$ has a `very small' colimit-dense subcategory:

\begin{prop}\label{colimdense} 
A set of power $3$ is colimit-dense in $\Set^{\op}$.
\end{prop}
\begin{proof}
Every set of power at most $2$ can be expressed by using an equalizer of two endomaps of $\{0,1,2\}$.
	
	For every set $X$ of power at least $3$ we present a diagram $D$
in $\Set$ whose objects have power $3$ and whose limit is $X$.
Choose elements $t$ in $X$ and $s$ outside of $X$, and for every element $x \neq t$ of $X$ put $K_x = \{t,x,s\}$.
 Given a subset $Y$ of $X$ and an element $x \in Y$, denote by $f_{Y,x}: Y \to K_x $ the function mapping $x$ to itself and the rest to $t$. For every element $x \in X \setminus\{t\}$ let $p_x$ be an endomap of $K_x $ whose fixed points are  $x$ and $t$ but not $s$. 

 Objects of $D$ are all the above sets $K_x$ and all three-element subsets $Y = \{t,x,x'\}$ of $X$. The only connecting morphisms are $f_{Y,x}:Y \to K_x$  for all $Y = \{t,x,x'\}$ and all $p_x$ above.
We have a cone of $D$ consisting of the functions $f_{X,x}:X \to K_x$ and $g_Y: X \to Y$, mapping $x$ to $x$ if $x \in Y$, else to $t$. This is a limit cone. 

Indeed, it is sufficient to verify that given a compatible choice of elements of all objects of $D$:
$$
k_x \in K_x \text{  and  }  l_Y \in Y,
$$
there exists a unique $v \in X$ such that
$$
k_v =v \text{  and  } l_Y=v \text{  for all  } Y=\{t,x,v\},
$$
 while choosing $t$ in all other cases. Observe that due to the connecting
map $p_x$ we have $k_x \neq s$ for every $x$.

(a) In case that there exists $x \in X\setminus\{t\}$ with $k_x =x$, we put
$v=x$. If $Y$ contains $v$, use the fact that $f_{Y,v}(l_Y)=k_v=v$ to conclude
$l_Y=v$. If $z$ is distinct from $v$, then $k_z=t$: use $f_{Z,z}$ for $Z=\{t,v,z\}$.
And we also have $l_Y=t$ if $Y$ does not contain $v$: choose $z \in Y\setminus \{t\}$ and use
$f_{Y,z}$.

	(b) In case that $k_x=t$ for all $x$, put $v=t$. We have $l_Y=t$ for
every $Y$: use $f_{Y,x}(l_Y)= k_x$ for $x\in Y\setminus\{t\}$.

Unicity is clear: if $v$, $v'$ are distinct in $X\setminus\{t\}$, then $l_Y$ for $Y=\{t,v,v'\}$ demonstrates that they do not both have the above property.
\end{proof}

\begin{remark}\label{colimdense1}
(a) In contrast, sets of power at most $2$ (that is, $\Set_3$) do not form a colimit-dense subcategory of $\Set^{\op}$. Indeed, a non-empty limit of any diagram $D: \cd \to \Set_3$ in $\Set$ always has cardinality which is a power of $2$. To see this, let $\cd_1$ be the small category obtained by adjoining to $\cd$ the formal inverses of all morphisms $u$ in $\cd$ for which $Du$ is an isomorphism. Then the diagram $D$ can be extended to a diagram $D_1:\cd_1\to\Set_3$ in the obvious way, and the limit of $D_1$ is the same as that of $D$. Let $\cd_0$ be the full subcategory of $\cd_1$ on objects $d$ for which there exists no morphism $u:d' \to d$ in $\cd_1$ such that $D_1(u)$ is a constant mapping. If the limit of $D$ is not empty, then the domain restriction $D_0: \cd_0 \to \Set$ of the diagram $D_1$ has the same limit as $D_1$ and $D$. All connecting morphisms of $D_0$ are isomorphisms between $2$-element sets. Thus every connected component of $\cd_0$ yields a subdiagram with a limit which is either empty or a two-element set. If $\cd_0$ has a connected component with empty limit, then the limit of $D$ is empty. Otherwise, $\cd_0$ has $k$ connected components with nonempty limits, thus, the limit of $D$ has cardinality $2^k$.

(b) $\Set_3$ is colimit-dense in the full subcategory $\ck$ of $\Set^{\op}$ on sets that have cardinality $2^k$ or $0$. Assuming 
$\neg$(M), $\ck$ is not bounded. Indeed, it is clear that $\Set$ is the idempotent completion of $\ck$: every
nonempty set $X$ is a retract of $2^X$. And if an idempotent completion $\ca$ of a category $\ck$ is not bounded, then $\ck$ is also not bounded. (Suppose, to the contrary, that $\cl$ is a small dense subcategory of $\ck$. Then we have a canonical full embedding of $\ck$
to the category of presheaves on $\cl$. It follows that also $\ca$ canonically embeds into that presheaf category,
thus $\cl$ is dense in $\ca$, a contradiction.)
\end{remark}

\begin{conclusion}\label{colimdense2}
Assuming $\neg$(M), $\Set^{\op}$ has a colimit-dense subcategory but not a dense one. Assuming (M), Vop\v enka's principle does not
hold and \ref{ex2} yields a cocomplete category having a colimit-dense subcategory but not a dense one.
\end{conclusion}

\section{Vector spaces}
We now turn to the category $\Vec$ of vector spaces over a given field $K$. There are numerous analogies to $\Set$, but there are also differences.
Let us start with the latter: Whereas, as we have seen, finite sets are colimit-dense in $\Set^{\op}$, finite-dimensional spaces are not colimit-dense
in $\Vec^{\op}$. To verify this,
recall that the dualization functor $(-)^\ast \colon \Vec \to ( \Vec)^{\op}$ given by $X^* = [X,K]$ (the space of all linear forms)  is left adjoint to its opposite, and this leads  to a monad on $\Vec$,
$$
\big( (-)^{\ast\ast}, \eta, \mu\big)
$$
called the \textit {double-dualization monad}. It assigns to every space $X$ its double-dual $X^{**}= [X^* , K]$ and to a
morphism $f\colon X\to Y$ the morphism $f^{\ast\ast} \colon X^{\ast\ast}\to Y^{\ast\ast}$ which takes $x\in X^{\ast\ast}$ to the linear map
$$
f^{**}(x) \colon Y^\ast \to K\,, \quad u\mapsto x(u\cdot f)\quad \mbox{for all }\ \ u\colon Y\to K\,.
$$
The unit has components
$$
\eta_X : X \to X^{\ast\ast} \,,\quad x\mapsto \operatorname{ev}_x
$$
where $\operatorname{ev}_x$ evaluates each $u\colon X\to K$ at $x$. We thus call the vectors of $X^{\ast\ast}$ of the form $\eta (x)$ the \textit{evaluation vectors}.

\begin{example}\label{LP}
A vector space $A$ of dimension $\aleph_0$ is not a limit of a diagram of finite-dimensional spaces in $\Vec$. This follows from the fact that $A$ is not isomorphic to the dual of any space
	$Y$ (since if $Y$ has dimension $n\ge\aleph_0$, then $Y^*$ has dimension $|K|^n$). Given a diagram $D: \cd \to \Vec$, let $D^{**}: \cd
	 \to \Vec$ denote the 
	composite of $D$ and $(-)^{**}$. If the objects of $D$ are finite-dimensional, we see that $D^{**}$ is naturally isomorphic to $D$. Therefore, the limit $lim D$ in $\Vec$ is the dual space to $colim D^*$, since $(-)^*$ takes colimits to limits. Thus that limit is not 
	isomorphic to $A$.
\end{example}

\begin{prop}\label{LP1}
All vector spaces of countable dimension form a colimit-dense subcategory of $\Vec^{\op}$.
\end{prop}
\begin{proof} For every infinite cardinal $n$, we construct a diagram $D$ in $\Vec$ whose objects have dimesion $1$ or $\aleph_0$ and whose limit is the $n$-dimensional space $V$ of all functions of finite support from $n$ to $K$.
Our diagram $D$ has as objects
\begin{enumerate}        
\item[(a)] the subspace $K_x$ of $V$, for $x \in n$, of all functions whose support is a subset of $\{x\}$, and
\item[(b)] the subspace $L_Y$ of $V$, for every countably infinite subset $Y$ of $n$, of all functions whose support is a subset of $Y$.
\end{enumerate}
For every $Y$ in (b) and every element $x \in Y$ we have the linear map $f_{Y,x}: L_Y \to K_x$ of domain restriction to $\{x\}$. These are precisely all the connecting maps of our diagram. We claim that $V$ is a limit of $D$ w.r.t.\ to the following cone: $g_x: V \to K_x$, domain restriction to $\{x\}$, and $g_Y: V \to L_Y$, domain restriction to $Y$.
It is easy to see that this is indeed a cone of $D$. 

Let another cone be given by a space $T$ and linear maps $h_x: T \to K_x$ and 
$h_Y: T \to L_Y$. Let $h: T \to K^n$ have components $h_x$ ($x \in n$) (using the obvious isomorphism of $K_x$ and $K$). Choose an arbittrary object $L_Y$.  Due to compatibility, we know for every element $x \in Y$ that 

$$
 h_x =f_{Y,x}\cdot h_Y.
$$
Therefore, we get a commutative square as follows
$$
\xymatrix{
T \ar [r]^{h} \ar [d]_{h_Y}& K^n \ar [d]^{p_Y}\\
L_Y\ar [r]_{e_Y}& K^Y
}
$$
where $p_Y$ is the projection and $e_Y$ the canonical embedding. (Indeed, by post-composing this square with the projection of $K^Y$ corresponding to $x \in Y$, we get the equality above.) This implies that for every element $t$ of $T$ the restriction of the function $h(t):n \to K$ to $Y$ has finite support. Since this holds for all countable subsets $Y$ of $n$, we conclude that
$h(t)$ has finite support for every $t \in T$. In other words, $h$ has a codomain restriction $h': T \to V$. This is the
desired factorization of the given cone. Indeed the equality
$$
h_Y = g_Y \cdot h'
$$
follows from the above square. Combined with $h_x = f_{Y,x}\cdot h_Y$ above, this implies
$$
h_x = g_x \cdot h'.
$$
The unicity of $h'$ is obvious.

\end{proof}
 
\begin{theorem}
	$\Vec^{\op}$ is bounded iff (M) holds.
\end{theorem}

It seems curious that this result has not been published before: the sufficiency has already been proved by Isbell in 1964: 
\cite{I1} Theorem 8.1. For the necessity, see \cite{R} Theorem 1.6.

\medskip

By a \textit{finite-dimensional linear partition} of a vector space $X$ we mean a surjective linear map $a \colon X \twoheadrightarrow  K^n$, where $n\in \N$. Every vector $t\in X^{\ast\ast}$ yields a choice of a member of the partition $a$ (or, equivalently, a choice of a vector $\alpha(a)$ of $K^n$): if $n=1$ we have  $a\in X^\ast$ and  the choice is simply $t(a) \in K$. In general, $a$ has components $a_1, \dots , a_n \in X^\ast$ and we put 
 $$
 \alpha (a) =\big(t(a_1), \dots , t(a_n)\big) \in K^n\,.
 $$
 This choice is \textit{coherent} in the expected sense: for every commutative triangle in $\Vec$
 $$
\xymatrix{  
& X\ar@{->>}[dl]_{b}
\ar@{->>}[dr]^a&\\
K^m\ar[rr] _u && K^n
}
$$
we have
$$
\alpha (a) = u\big(\alpha(b)\big)\,.
$$
Indeed, $u$ is a matrix $(u_{ij})$ and $a$ has  components $a_i=\sum\limits_{j} u_{ij} b_j$. Since $t$ is a linear map, $\alpha(a)$ has components
$$
t(a_i) =\sum u_{ij}t(b_j)\,.
$$
 This is precisely the $i$-th component of $u\big(t(b_1),\dots , t(b_m)\big)$.

\begin{lemma}\label{vect}
For every space $X$ the vectors of $X^{\ast\ast}$ are precisely the coherent choices of a member of every finite-dimensional linear partition of $X$. That is, the above passage $t\mapsto \alpha$ is bijective.
\end{lemma}

\begin{proof}
Let $\alpha$ be a coherent choice. We prove that there exists a unique $t\in X^{\ast\ast}$ with $\alpha(a) =\big(t(a_1),\dots , t(a_n)\big)$ for every $a= \langle a_1, \dots , a_n\rangle \colon X \twoheadrightarrow K^n$.

Given $a\in X^\ast$, then either $a=0$ or $a\colon X \twoheadrightarrow K$ is surjective. We define $t\in X^{\ast\ast}$ by
$$ 
t(0) =0 \quad \mbox{and}\quad t(a) =\alpha(a) \quad \mbox{for}\quad a\ne 0\,.
$$

(1) $t$ is linear. Indeed, to prove $t(ka) = kt(a)$ for  every $k\in K$, we can restrict ourselves to $k\ne 0$. Then  $ka$ is surjective whenever $a$ is. And we have a commutative triangle in $\Vec$ as follows
$$
\xymatrix{  
& X\ar[dl]_{a}
\ar[dr]^{ka}&\\
K\ar[rr]_{k\cdot(-)} && K
}
$$
Since $\alpha$ is coherent, this yields $t(ka) = kt(a)$.

To prove $t(a_1+a_2) = t(a_1) + t(a_2)$ , we can clearly assume $a_i\ne 0$ and also $a_1+a_2 \ne 0$ (for the case $a_1 = -a_2$ use $k=-1$ above). Let $\langle a_1, a_2\rangle \colon X \to K^2$ have the image factorization $j \cdot e$ where $e\colon X \to A$ for $A =K$ or $K^2$
is surjective and $j$ is injective. The following triangles
$$
\xymatrix{  
& &X\ar@{->>}[dll]_{a_1}
\ar@{->>}[drr]^{a_2}
\ar@{->>}[d]^e&&\\
K&K^2\ar[l]^{\pi_1}& A \ar[l]^j \ar[r]_j& K^2\ar[r]_{\pi_2} &K
}
$$
imply $\alpha(a_i) = \pi_i \cdot j\cdot \alpha(e)$, therefore
$$
\alpha(a_1)+ \alpha(a_2) =\oplus \cdot j \cdot \alpha(e)
$$
for the addition $\oplus \colon K^2 \to K$. And the following triangle
$$
\xymatrix{  
& X\ar@{->>}[dl]_{e}
\ar[dr]^{a_1+a_2}&\\
A\ar[r]_j &K^2\ar[r]_{\oplus} & K
}
$$
yields $\alpha(a_1)+ \alpha(a_2) =\alpha(a_1+a_2)$. This proves that $t$ is linear.

(2) $t$ satisfies $\alpha(a) =\big(t(a_1), \dots , t(a_n)\big)$. This is clear for $n=1$. For general $n$, given $a_i\ne 0$ use the coherence of $\alpha$ on the triangle
$$
\xymatrix{  
& X\ar[dl]_{a}
\ar[dr]^{a_i}&\\
K^n\ar[rr]_{\pi_i}&& K
}
$$

(3) $t$ is unique -- this is obvious.
\end{proof}

For every infinite cardinal $\lambda$ denote by
$$
\Vec_\lambda
$$
the full subcategory of $\Vec$ on spaces of dimension less that $\lambda$.
Recall from the Introduction Leinster's result that the full embedding
$$
\Vec_{\omega} \hookrightarrow  \Vec
$$
of finite-dimensional spaces has the double-dualization monad as the  codensity monad. Our Lemma allows for a direct proof.
B{\"{o}}rger's result mentioned above about the unique monad structure  on the ultrafilter functor $\du$ is based on the fact, proved in \cite{B}, that $\du$ is the terminal object of the category of set functors preserving finite coproducts. We now prove an analogous fact about $(-)^{\ast\ast}$. A functor  $F\colon \Vec \to\Vec $ is called  \textit{linear} if the induced maps $\Vec (X,Y) \to \Vec (FX, FY)$ are linear for all spaces $X,Y$.

\begin{lemma}\label{vect1} 
For every linear endofunctor $F$ of $\Vec $ with $FK\cong K$ there exists a natural transformation $\alpha\colon F\to (-)^{\ast\ast}$. It is unique up to a scalar multiple, i.e., every other such transformation has the form $k\alpha$ for some $k\in K$.
\end{lemma}

\begin{proof}
	Without loss of generality assume $FK=K$.
	
	(1) Existence. For every space $X$ define a function
	$$
	\alpha_X \colon FX \to X^{\ast\ast}
	$$
	as follows: given $x\in FX$, then $\alpha_X(x) \colon X^\ast \to K$ is defined by 
	$$
	\alpha_X(x)(t) = Ft(x) \quad \mbox{for all} \quad t\colon X\to K\,.
	$$
	Since  $F$ is a linear functor, for every vector $x$ of $X$ the map $\alpha_X(x)(-)$ is linear. And since each $Ft(-)$ is a linear map, $\alpha_X$ is linear.
	Let us verify the naturality squares
	$$
	\xymatrix@C=3pc{
		FX \ar[d]_{Ff} \ar[r]^{\alpha_X} & X^{\ast\ast}\ar[d]^{f^{\ast\ast}}\\
		FY \ar[r]_{\alpha_Y} & Y^{\ast\ast}
	}
	$$
	The upper passage applied to $x\in FX$ yields $\alpha_X(x) \cdot f^\ast$. To every $s\colon Y\to K$ this function assigns
	$$
	f^{\ast\ast} \big(\alpha_X(x)\big) (s) = \alpha_X(x) (s\cdot f) = F(s\cdot f) (x).
	$$
	The lower passage assigns to $s$ the value
	$$
	\alpha_Y\big(Ff(x)\big)(s) = F s  \big(Ff(x)\big)\,, 
	$$
	which is the same one.
	
	(2)
	Uniqueness. Let $\beta \colon F\to (-)^{\ast\ast}$ be a  natural transformation. The component
	$$
	\beta_K \colon K \to K^{\ast\ast} \quad (\cong K)
	$$
	is a linear map, hence, it is given by a scalar $k\in K$ in the sense that
	$$
	\beta_K (l) (u) = u(kl) \quad \mbox{for all}\quad l\in K, \ u\colon K\to K\,.
	$$
	We are going to prove that $\beta =k\alpha$, i.e., for all $x\in FX$ and $t\colon X\to K$ we have
	$$
	\beta_X (x)(t) = k\cdot Ft(x)\,.
	$$
	This follows  from the naturality square
	$$
	\xymatrix@C=3pc{
		FX \ar[d]_{Ft} \ar[r]^{\beta_X} & X^{\ast\ast}\ar[d]^{t^{\ast\ast}}\\
		K \ar[r]_{\beta_K} & K^{\ast\ast}
	}
	$$
	We apply it to $x\in FX$ and get $t^{\ast\ast}(\beta_X(x))$ which to $\id_K \in K^\ast$ assigns the value $\beta_X(x)(t)$.  The lower passage applied to $x$ assigns to $\id_K$ the value $\beta_K (Ft(x)) =k\cdot Ft(x)$.
\end{proof}

\begin{corollary}\label{vect2}
	Let $((-)^{\ast\ast}, \eta, \mu)$ be the double-dualization monad. Then every monad structure on the endofunctor $(-)^{\ast\ast}$ has, for some scalar $k\in K\setminus\{0\}$, the form $((-)^{\ast\ast}, k\eta, k^{-1}\mu)$.
\end{corollary}

Indeed, the above lemma implies that the unit is $k\eta$ and the multiplication is $l\mu$ for $k$, $l\in K$. From the unit law $(l\mu)\cdot [(-)^{\ast\ast} (k\eta)]=\id$, i.e., $lk \mu \cdot [(-)^{\ast\ast}\eta]=\id$, we deduce $lk=1$.

\begin{corollary}\label{vect3}
	The codensity monad of the embedding
	$$
	\Vec_\omega \hookrightarrow  \Vec
	$$
	is the double-dualization monad.
\end{corollary}

This is analogous to Corollary 3.7 for $\du$: we first verify that if  $(T, \eta^ T, \mu ^T)$ is the codensity monad, then the endofunctor $T$ can be chosen to be $(-)^{\ast\ast}$. This follows from Remark \ref{colimdense1} just as for $\du$ above: recall that limits in $\Vec$ are created by  the forgetful functor to $\Set$. Thus $T$ assigns to $X$ a limit  of the diagram $D_X$ of all $a\colon X\to K^n$, $n\in \N$, which consists of all compatible choices of elements of $K^n$ for all $a$'s. And Lemma~\ref{vect} allows us to put $TX = X^{\ast\ast}$ with limit projections
$$
\pi_a : X^{\ast\ast} \xrightarrow{\ a^{\ast\ast}\ } (K^n)^{\ast\ast} \cong K^n\,,
$$
for all $a$. Given a morphism $f\colon X\to Y$, then $Tf$ is defined  by $ \pi_a\cdot\nobreak Tf =\pi_{a\cdot f}$, and it is easy to see that $f^{\ast\ast}$ satisfies these equalities, thus $Tf = f^{\ast\ast}$.

Next  recall that the monad unit $\eta^T \colon \operatorname{Id} \to T$ has, according to the limit formula, components $\eta_X\colon X\to TX$ determined by the commutative triangles below
\begin{center}
	$
	\xymatrix{
		X \ar [rr]^{\eta^T_X} \ar[dr]_a && TX\ar[dl]^{\pi_a}
		\\
		& K^n & }
	$
	\quad for all $a\colon X\to K^n$
\end{center}
Due to the above choice of the limit cone $\pi_a$, we see that the unit of $(-)^{\ast\ast}$ satisfies $\pi_a\cdot \eta_X =a$ holds for all $a$, thus, $\eta = \eta^T$. This means that for $(\eta^T, \mu^T)$ the scalar in the preceding lemma is $k=1$. 

\begin{definition}\label{vect4}
Let $\lambda$ be an infinite cardinal.

By a \textit{linear $\lambda$-partition} of a space  $X$ we mean a surjective linear map onto a space of dimension less that $\lambda$.

		A vector $x$ of $X^{\ast\ast}$ is called 
 \textit{$\lambda$-complete}  if for every linear $\lambda$-partition $a\colon X\to A$ we have: $a^{\ast\ast} (x)$ is an evaluation vector (i.e., it lies in the image of $\eta_A$).
	
\end{definition} 

\begin{lemma}\label{vect5}
	Let $\lambda$ be an infinite cardinal. For every space $X$ the $\lambda$-complete vectors of $X^{\ast\ast}$  are precisely the coherent choices of members of linear $\lambda$-partitions of $X$.
\end{lemma}

\begin{proof}
First recall that all linear forms of a given space are collectively monic. Indeed, given two distinct vectors $p$ and $q$, we can assume $p$ non-zero and choose a basis
containing $p$. Moreover, in case $q$ is not a scalar multiple of $p$ we include $q$
in that basis. The linear form assigning to every vector its $p$-coordinate sends $p$ to $1$ and $q$ to a different value.

For every coherent choice $\alpha$ as above it is our task to show that there exists
a vector $x$ of $X^{\ast\ast}$  with $\alpha(a)=a^{\ast\ast}(x)$ for every linear map $a: X\to A$ with $dimA<\lambda$. Since $\alpha$ yields, in particular, a choice for all finite-dimensional  linear partitions, we know from the above lemma that an $x$
exists such that $\alpha(a)=a^{\ast\ast}(x)$ holds for all linear forms $a$. Given a linear $\lambda$-partition, we prove {$\alpha(a)=a^{\ast\ast}(x)$} by verifying that
every linear form $u:A \to K$ merges both sides. By coherence, 
\[
u(\alpha(a))=\alpha(u\cdot a)=(u\cdot a)(x).
\]
And this is precisely the value of $u$ at $a^{\ast\ast}(x)$.
\end{proof}

\begin{notation} For every infinite cardinal $\lambda$ denote by
$(-)^{**}_\lambda$ the subfunctor of the double-dualization monad assigning to every space all $\lambda$-complete vectors of its double dual.
\end{notation}

\begin{corollary} The codensity monad of the emebdding
	$$
	\Vec_\lambda \hookrightarrow  \Vec
	$$
is the submonad of the double-dualization monad carried by $(-)^{**}_\lambda$.
\end{corollary}

The proof is completely analogous to that of Corollary \ref{vect3}.The following is the analogue of Corollary \ref{isbell}. 

\begin{theorem}	
	The category $\Vec^{\op}_\lambda$ is dense in $\Vec^{\op}$ iff every $\lambda$-complete vector of $X^{\ast\ast}$ is an evaluation vector (for all spaces $X$).
\end{theorem}

\begin{proof}
	For every space $X$ let $D_X\colon X\big/\Vec_\lambda\to\Vec$ be the diagram assigning  to each $a\colon X\to A$ the codomain. The limit of $D_X$ is the subspace of $X^{\ast\ast}$ formed by all $\lambda$-complete vectors. To see this, recall that limits are created by the forgetful functor to $\Set$. Thus, $\lim D_X$ can be described  as the space of all compatible choices $(\hat-)$ of elements  $\hat a\in A$ for all $a\colon X\to A$ with $\dim A<\lambda$. By the preceding lemma,
	we conclude that if $\lambda$ has the property in our theorem, then $\Vec_\lambda$ is codense in $\Vec$. Indeed, for every space $X$ the canonical limit of $D_X$ is $\eta_X [X] \cong X$.
 The converse implication is a particular case of the next proposition.
\end{proof}

\begin{prop}
 If $\cb$ is a small codense subcategory in $\Vec$ and $\lambda$ is an infinite cardinal larger than the dimensions of all spaces of $\cb$, then,
for all spaces $X$, every $\lambda$-complete vector of $X^{\ast\ast}$ is an evaluation vector.
\end{prop}

\begin{proof}
 The category $\Vec_\lambda$ clearly contains $\cb$ as a cofinal subcategory. Denote by $\bar D_X \colon X\big/\cb\to \Vec$ and $ D_X \colon X\big/\Vec_\lambda\to \Vec$ the canonical diagrams. They have the same cones. More precisely, every cone of $D_X$ restricts to one of $\bar D_X$, and every cone of $\bar D_X$ can, via cofinality, be uniquely extended to one of $D_X$. Hence, $\lim \bar D_X =\lim D_X$. Thus, for every space $X$ the limit of $D_X$, consisting of all $\lambda$-complete vectors of $X^{\ast\ast}$, yields $X\cong \eta[X]$. Therefore, $\lambda$ has the property of our theorem and proposition.
\end{proof}

 \end{document}